\documentclass[11pt]{article}
\usepackage{amsfonts}
\usepackage{color}
\usepackage{amsmath,amssymb,comment}
\usepackage{doi}
\usepackage{hyperref}
\usepackage[english]{babel}
\newtheorem{theo}{Theorem}[section]
\newtheorem{lem}[theo]{Lemma}

\newtheorem{prop}[theo]{Proposition}

\newcommand{\mysection}[1]{\section{#1} \setcounter{equation}{0}}
\newcommand{\proof}{{\sc Proof.} \quad}
\newcommand{\proofc}{{\sc Proof} \ }
\newcommand{\be}{\begin{equation} \label}
\newcommand{\ee}{\end{equation}}
\newcommand{\bea}{\begin{eqnarray}\label}
\newcommand{\eea}{\end{eqnarray}}
\newcommand{\bas}{\begin{eqnarray*}}
\newcommand{\eas}{\end{eqnarray*}}
\newcommand{\bit}{\begin{itemize}}
\newcommand{\eit}{\end{itemize}}
\newcommand{\qed}{\hfill$\Box$ \vskip.2cm}
\newcommand{\nn}{\nonumber}
\newcommand{\R}{\mathbb{R}}

\newcommand{\pO}{\partial\Omega}

\newcommand{\eps}{\varepsilon}

\newcommand{\io}{\int_\Omega}
\newcommand{\na}{\nabla}
\newcommand{\Del}{\Delta}

\newcommand{\al}{\alpha}

\newcommand{\pa}{\partial}
\newcommand{\bom}{\overline{\Omega}}
\newcommand{\Om}{\Omega}

\newcommand{\vs}{\vspace*}
\newcommand{\hs}{\hspace*}

\newcommand{\vp}{\varphi}
\newcommand{\lbal}{\left\{ \begin{array}{l}}
\newcommand{\lball}{\left\{ \begin{array}{ll}}
\newcommand{\ear}{\end{array} \right.}

\newcommand{\abs}{\\[5pt]}

\newcommand{\adb}{\allowdisplaybreaks}
\renewcommand{\div}{{\rm div} \,}

\newcommand{\tm}{T_{max}}

\newcommand{\ugamma}{\underline\gamma}
\newcommand{\ogamma}{\overline\gamma}

\usepackage[
  style=numeric,        
  sorting=nyt,          
  doi=true,             
  giveninits=true,
  url=false,            
  isbn=false,           
  eprint=false,
  backend=biber
]{biblatex}

\DeclareBibliographyDriver{book}{%
  \usebibmacro{bibindex}%
  \usebibmacro{begentry}%
  \usebibmacro{author/editor+others/translator+others}%
  \setunit{\labelnamepunct}\newblock
  \usebibmacro{title}%
  \newunit\newblock
  \usebibmacro{byeditor+others}%
  \newunit\newblock
  \printfield{edition}%
  \newunit\newblock
  \printlist{publisher}%
  \newunit\newblock
  \printfield{year}%
  \newunit\newblock
  \iffieldundef{url}{}{\printfield{url}\newunit\newblock}%
  \usebibmacro{finentry}%
}

\DeclareNameAlias{default}{family-given}

\newcommand{\ds}{\delta_\star^{\frac 12}}
\newcommand{\dse}{\delta_\star^{\beta}}
\begin{document}
\adb
%
%
\title{ Large time existence in a thermoviscoelastic evolution problem with mildly 
temperature-dependent parameters}
\author{
Felix Meyer\footnote{felix.meyer@math.uni-paderborn.de}\\
{\small Universit\"at Paderborn, Institut f\"ur Mathematik}\\
{\small 33098 Paderborn, Germany} }
\date{}
\maketitle
\begin{abstract}
\noindent 
We consider
\bas \label{HS}
	\lbal
	u_{tt} = (\gamma(\Theta) u_{xt})_x + a (\gamma(\Theta) u_x)_x +(f(\Theta))_x, \\[1mm]
	\Theta_t = D\Theta_{xx} + \Gamma(\Theta) u_{xt}^2 + F(\Theta) u_{xt}, 
	\ear
    \qquad \qquad (\star)
\eas
under Neumann boundary conditions for $u$ and Dirichlet boundary conditions for $\Theta$ in a bounded interval  $\Om\subset\R$. \abs
This model is a generalization of the classical system for the
description of strain and temperature evolution in
a thermo-viscoelastic material following a Kelvin-Voigt material law, in which $\gamma\equiv \Gamma$ and $f\equiv F$. Different variations of this model have already been analyzed in the past and the present study draws upon a known result concerning the existence of classical solutions, which are local in time, for suitably smooth initial data, arbitrary $a>0$, $D>0$ and $\gamma,f\in C^2([0,\infty))$ as well as $\Gamma,F\in C^1([0,\infty))$ with $\gamma>0,\Gamma\ge0$ and $F(0)=0$. \abs 
Our work focuses on proving that existence times for classical solutions can be arbitrarily large, assuming sublinear temperature dependencies of $\gamma$ and $f$, and further $|F(s)|\le C_F(1+s)^\alpha$ for some $C_F>0$ and $\alpha\in(0,1)$. 
In particular, for any given $T_\star$, initial mass $M$ and $0<\ugamma<\ogamma$, there exists a constant $\delta_\star(M,T_\star,a,D,\Om,\ugamma,\ogamma,C_F,\alpha)>0$, such that if 
$$\ugamma\le\gamma\le\ogamma \quad\mbox{ and }\quad 0\le \Gamma\le \ogamma \quad \mbox{ as well as } \quad\|\gamma'\|_{L^\infty([0,\infty))}\le \delta_\star \quad \mbox{ and }\quad  \|f'\|_{L^\infty([0,\infty))}\le \delta_\star $$
hold, the maximal existence time of the classical solution to $(\star)$ surpasses $T_\star$.
Therefore, converting $(\star)$ into a parabolic system using the substitution $v:=u_t+au$ is key to applying known methods from works on parabolic problems.

\end{abstract}
\newpage
\section{Introduction}\label{intro}

We consider the problem
\be{0}
	\lball

	u_{tt} = (\gamma(\Theta) u_{xt})_x + a (\gamma(\Theta) u_x)_x + (f(\Theta))_x, 
	\qquad & x\in\Om, \ t>0, \\[1mm]
	\Theta_t = D\Theta_{xx} + \Gamma(\Theta) u_{xt}^2+  F(\Theta)u_{xt},
	\qquad & x\in\Om, \ t>0, \\[1mm]
	\frac{\pa u}{\pa\nu}=0,\  \Theta(x,t)=0
	\qquad & x\in\pO, \ t>0, \\[1mm]
	u(x,0)=u_0(x), \quad u_t(x,0)=u_{0t}(x), \quad \Theta(x,0)=\Theta_0(x),
	\qquad & x\in\Om,
	\ear
\ee
which models the heat generation of acoustic waves in Kelvin-Voigt type materials on an open interval $\Om\subset \R$. The model establishes a correlation between the temperature variable, designated as $\Theta$, and the mechanical displacement variable, designated as u. This correlation is realised by means of feedback functions $\gamma$ and $f$, which describe the mechanical losses according to the Kelvin-Voigt material laws (\cite{GutierrezLemini2013}, \cite{Chawla2008}, \cite{Boley1960}). We consider $a$
 and $D$ to be positive parameters as well as $u_0,u_{0t}$ and $\Theta_0\ge0$ as given suitably smooth functions. \\
For thermoviscoelastic evolution in Kelvin-Voigt materials of arbitrary spatial dimensions, \eqref{0} can be regarded as a simplified version of 
\be{01}
	\lbal
	u_{tt} =  d\div (\gamma(\Theta) : \na^s u_t) + a\div (\gamma(\Theta) : \na^s u) + \div f(\Theta), \\[1mm]
	\zeta(\Theta)\Theta_t = D\Del\Theta + d \langle \gamma(\Theta) : \na^s u_t , \na^s u_t \rangle + \langle F(\Theta) , \na^s u \rangle,
	\ear
\ee
where $u$  is considered as a vector field on a bounded domain $\Om\subset\R^n$, where $\nabla^s u=\frac 12 (\nabla u + (\nabla u)^T)$ denotes the corresponding symmetric gradient, and where $\gamma$, $f$ and $F$ are given tensor- and matrix-valued functions. For a more detailed derivation we refer to \cite{Roubicek2009} and \cite{Claes2024}.

For variations of the more general model in the higher dimensional version \eqref{01}, the analysis of solvability on a quite basic generalized level causes such sustainable challenges that even for a temperature-
independent feedback function $\gamma$, results have been achieved only under quite restrictive assumptions e.g. on
the growth of $f$ (\cite{Blanchard2000}, \cite{Roubicek2009}) and the smallness of the initial data (\cite{Shibata1995}, \cite{Racke_1997}), or they rely on temporally local
solutions (\cite{Jiang_1990}, \cite{Bonetti2003}) and solutions in a generalized framework (\cite{Claes2024}). \\
See also \cite{Owczarek2023}, \cite{Pawlow2017}, \cite{Rossi2013}, \cite{Gawinecki2016}, \cite{Gawinecki2016a}, \cite{Cieslak2023} or \cite{Bies2023} for more results in related models.\abs

For a temperature dependent $\gamma$, even for a quasilinear variation of \eqref{01}, where $\gamma$ and $\Gamma$ are considered to be mapping from $[0,\infty)$ to $\R$ global existence of classical solutions has been established, but only by assuming suitably smooth and small initial data (\cite{Claes}). \abs

In the context of spatially one-dimensional variations of \eqref{01}, the existence of global weak (\cite{Racke_1997}) and even classical solutions has been proven, but only for $\gamma\equiv\mbox{const.}$ (\cite{Dafermos_1982}) or $\gamma\equiv\gamma(u)$
(\cite{Kawohl_1985}, \cite{Jiang_1993}). 
In the event of temperature dependencies being taken into consideration, at least in some components, the following system occurs, among others
\bas \label{HS1}
	\lbal
	u_{tt} = (\gamma(\Theta) u_{xt})_x + a u_{xx} -(f(\Theta))_x, \\[1mm]
	\Theta_t = \Theta_{xx} + \gamma(\Theta) u_{xt}^2 - f(\Theta) u_{xt}, 
	\ear
\eas

for which the existence of global weak solutions was proven even under weaker requirements on the initial data. In particular, arbitrarily large initial data were allowed and, in addition to a growth assumption on $f$, only the boundedness of $\gamma$ was required (\cite{Winklera}).

Under the stricter assumption on the boundedness of $\gamma$, that even
$\gamma_0 \le \gamma \le \gamma_0+\delta$ 
is valid for a $\delta>0$, which can be controlled, not only the existence of global classical solutions, but also the existence of stronger solutions under weaker requirements on the initial data has been established (\cite{Winkler}).
For a related system, the existence of global classical solutions was also proven for small and smooth initial data (\cite{Slemrod_1981}, \cite{Fricke2025}).

In line with this, in the present manuscript we will focus on situations where the initial data are suitably smooth, but not necessarily small and the temperature dependencies
of the core ingredients to (\ref{0}) are present but moderate.
This aligns with recent experimental works
which have revealed that in some piezoceramic materials, the elastic parameters indeed exhibit some mild but relevant
variations with respect to temperature (\cite{Friesen2024}).
Accordingly, we will examine how far appropriate assumptions on smallness of the derivatives $\gamma'$ and $f'$ 
may warrant that solutions exist at least up to some prescribed time horizon $T_\star$.\abs
Our considerations in this regard will be prepared by the following basic statement on local-in-time solvability
and extensibility, as can be obtained by a 
straightforward adaptation of a reasoning detailed in \cite{Claes} for the variant
of (\ref{0}) in which homogeneous Neumann boundary conditions are imposed for both solution components. As there, the original wave model is converted into a parabolic problem using the substitution $v:=u_t+au$. This is not only essential in the proof mentioned above, but will also come into play in our later analysis.

\begin{prop}\label{theo_loc}
  Let  $\Om\subset\R$ be an open interval, suppose that
  \be{Vorr}
	\lbal
	a>0 \mbox{ and } D>0 \mbox{ are constants, that }
	\\[1mm]
F\in C^1([0,\infty))\mbox{ satisfies } F(0)=0,\mbox{ and } |F(s)|\le C_F (1+s)^\alpha 
	\ear
  \ee 
  for some $C_F>0$ and $\alpha\in(0,1)$ and
\be{Vorr2}
	\lbal
	\gamma \in C^2([0,\infty)) \mbox{ and } \Gamma\in C^1([0,\infty))  \mbox{ are such that } \gamma>0 \mbox{ and }\Gamma\ge 0 \mbox{ on }[0,\infty), \mbox{ and that }\\[1mm]
f\in C^2([0,\infty)) 
	\ear
  \ee 
  
  hold and that
  \be{Init}
	\lbal
	u_0\in C^2(\bom)
	\mbox{ is such that $\frac{\pa u_0}{\pa\nu}=0$ on $\pO$,} \\[1mm]
	u_{0t}\in C^{1+\mu}(\bom)			
	\mbox{ is such that $\frac{\pa u_{0t}}{\pa\nu}=0$ on $\pO$, \qquad and} \\[1mm]
	\Theta_0\in C^{1+\mu}(\bom)	
	\mbox{ satisfies $\Theta_0\ge 0$ in $\Om$ and $\frac{\pa \Theta_0}{\pa\nu}=0$ on $\pO$,}
	\ear
  \ee
  with some $\mu\in (0,1)$.
  Then there exist $\tm\in (0,\infty]$ as well as functions
  \be{tl1}
	\lbal
	u\in \Big( \bigcup_{\eta\in (0,1)} C^{1+\eta,\frac{1+\eta}{2}}(\bom\times [0,\tm))\Big) \cap C^{2,1}(\bom\times (0,\tm))
		\qquad \mbox{and} \\[1mm]
	\Theta\in \Big( \bigcup_{\eta\in (0,1)} C^{1+\eta,\frac{1+\eta}{2}}(\bom\times [0,\tm))\Big) 
		\cap C^{2,1}(\bom\times (0,\tm))
	\ear
  \ee
  which are such that
  \be{tl2}
	\begin{array}{l}
	u_t\in 
	\Big( \bigcup_{\eta\in (0,1)} C^{1+\eta,\frac{1+\eta}{2}}(\bom\times [0,\tm))\Big) \cap C^{2,1}(\bom\times (0,\tm)),
	\end{array}
  \ee
  that $\Theta\ge 0$ in $\Om\times (0,\tm)$, that $(u,\Theta)$ solves \eqref{0} in the classical pointwise sense
  in $\Om\times (0,\tm)$, and which have the additional property that
  \bea{Ext}
	& & \hs{-15mm}
	\mbox{if $\tm<\infty$, \quad then \quad} \nn\\
	& & \hs{-6mm}
	\limsup_{t\nearrow\tm} \Big\{ \|u_t(\cdot,t)\|_{W^{1,2}(\Om)} + \|\Theta(\cdot,t)\|_{L^\infty(\Om)} \Big\}
		=\infty.	
  \eea
\end{prop}

\color{black}
\vs{4mm}
{\bf Main results: Large-time existence in the presence of slowly varying $\gamma$ and $f$.} \quad
\\
Our main result now states that given any $T_\star>0$, an assumption essentially requiring the smallness of both derivatives
$\gamma'$ and $f'$ ensures that the maximal existence time identified in Proposition \ref{theo_loc} actually satisfies
$\tm\ge T_\star$.
This will be achieved by tracing the evolution of the functionals
\bas
	y(t):= 1+ \delta_\star^{\frac 12} \io \Theta_x^2+\delta_\star^{\beta} \rho \io u_x^4 +\io v_{x}^2 + 4a^2 \io u_x^2
\eas
along trajectories, with respect to $v:=u_t+au$, our previously mentioned substitution, and where $\rho>0$ is a suitably chosen number depending on $a,D$ and an upper bound $\ogamma$ for $\gamma$,
and where $\delta_\star$ is an appropriately small constant which depends not only on the above system parameters and a lower bound $\ugamma$ for $\gamma$ but moreover on $T_\star$.
Based on an inequality of the form
\bas
	y'\le \delta^\kappa \cdot y^3+C\cdot y,
\eas
valid with some conveniently small $\delta>0$ and $\kappa>0$ as well as some $C>0$ that may be large but can be favorably controlled with respect to the above system parameters (see Lemma 3.5), we will utilize an ODE comparison argument to warrant that $T_\star\le \tm$.\abs
We will thereby achieve the following main result of this study.

\color{black}
\begin{theo}\label{theo12}
  Let $\Om\subset\R$ be an interval and suppose $a,D$ and $F$ are such that \eqref{Vorr} is satisfied for some positive constants $C_F$ and $\alpha\in(0,1)$ and that the initial data $(u_0, u_{0t},\Theta_0)$ to satisfy \eqref{Init}.
  Then for every $M$, $\ugamma$, $\ogamma$ and $T_\star>0$, there exists a constant $\delta_\star=\delta_\star(M,T_\star,a,D,\Om,\ugamma,\ogamma,C_F,\alpha)>0$ such that if $u_0, u_{0t}, \Theta_0, \gamma,\Gamma$ and $f$ satisfy not only $\eqref{Vorr2}$ but also
  \be{12.2}
	\io u_{0tx}^2 + \io  u_{0x}^2 + \io  u_{0x}^{4} + \io \Theta_{0x}^2 \le M,
  \ee
and
\be{gamBesch}
  \ugamma\le\gamma(s)\le \ogamma \qquad \mbox{ and }\quad 0\le \Gamma(s)\le \ogamma\quad
  \mbox{ on }\quad [0,\infty)
  \ee
   as well as 
  \be{12.3}
\|\gamma'\|_{L^\infty([0,\infty))}\le\delta_\star\quad \mbox{ and } \quad\|f'\|_{L^\infty([0,\infty))} \le \delta_\star,\ee
  it follows that in Proposition \ref{theo_loc} actually $\tm\ge T_\star$ holds.

\end{theo}{}
\mysection{Local solvability}
As in \cite{Claes} we formally substitute $v:=u_t+au$ to rewrite \eqref{0} in equivalent form as the parabolic problem
\be{0v}
	\lball
	v_t = (\gamma(\Theta)\na v)_x + av - a^2 u + (f(\Theta))_x,
	\qquad & x\in\Om, \ t>0, \\[1mm]
	u_t = v-au,
	\qquad & x\in\Om, \ t>0, \\[1mm]
	\Theta_t = D\Theta_{xx} + \Gamma(\Theta) |v_x - a u_x|^2+ F(\Theta)(v_x-au_x),
	\qquad & x\in\Om, \ t>0, \\[1mm]
	\frac{\pa v}{\pa\nu}=\frac{\pa u}{\pa\nu}=\Theta(x,t)=0,
	\qquad & x\in\pO, \ t>0, \\[1mm]
	v(x,0)=v_0(x), \quad u(x,0)=u_0(x), \quad \Theta(x,0)=\Theta_0(x),
	\qquad & x\in\Om,
	\ear
\ee

where for $a,D,f,F,\Gamma$ and $\gamma$ we assume that $\eqref{Vorr}$ and $\eqref{Vorr2}$ hold.   
 In view of the change to Dirichlet boundary conditions for $\Theta$, the proof will instead rely on Dirichlet semi-heatgroup theory, which leads to similar results, so no further adjustments will be necessary there. The proof in \cite{Claes} is based on a fixed-point approach, whereby the requirements of the applied Schauder theorem were met by utilizing scalar parabolic theory, which can be carried out analogously.  
\begin{lem}\label{lem_loc}
  Let $\Om\subset\R$ be a bounded interval and suppose that
  $a,D,\gamma,f$ and $F$ are such that \eqref{Vorr} and $\eqref{Vorr2}$ are satisfied.
  Then whenever $\mu\in (0,1)$ and 
  \be{init}
	\lbal
	v_0\in C^{1+\mu}(\bom)	
	\mbox{ such that $\frac{\pa v_0}{\pa\nu}=0$ on $\pO$,} \\[1mm]
	u_0\in C^2(\bom)			
	\mbox{ such that $\frac{\pa u_0}{\pa\nu}=0$ on $\pO$ \qquad and} \\[1mm]
	\Theta_0\in C^{1+\mu}(\bom)	
	\mbox{ such that $\Theta_0\ge 0$ in $\Om$ and $\Theta=0$ on $\pO$,}
	\ear
  \ee
  one can find $\tm\in (0,\infty]$ as well as
  \be{l1}
	\lbal
	v\in \Big( \bigcup_{\eta\in (0,1)} C^{1+\eta,\frac{1+\eta}{2}}(\bom\times [0,\tm))\Big) \cap C^{2,1}(\bom\times (0,\tm)),
		\\[1mm]
	u\in \Big( \bigcup_{\eta\in (0,1)} C^{1+\eta,\frac{1+\eta}{2}}(\bom\times [0,\tm))\Big) \cap C^{2,1}(\bom\times (0,\tm))
		\qquad \mbox{and} \\[1mm]
	\Theta\in \Big( \bigcup_{\eta\in (0,1)} C^{1+\eta,\frac{1+\eta}{2}}(\bom\times [0,\tm))\Big) 
		\cap C^{2,1}(\bom\times (0,\tm))
	\ear
  \ee
  such that $\Theta\ge 0$ in $\Om\times (0,\tm)$, and that $(v,u,\Theta)$ solves (\ref{0v}) classically 
  in $\Om\times (0,\tm)$. Moreover, $\tm$ can be chosen such that
  \bea{ext}
	& & \hs{-15mm}
	\mbox{if $\tm<\infty$, \quad then \quad} \nn\\
	& & \hs{-6mm}
	\limsup_{t\nearrow\tm} \Big\{ \|v(\cdot,t)-au(\cdot,t)\|_{W^{1,2}(\Om)} + \|\Theta(\cdot,t)\|_{L^\infty(\Om)} \Big\}
		=\infty.
  \eea
\end{lem}

Our result on local-in-time solvability and extensibility in (\ref{0}) has thereby in fact been established already:\abs
\proofc of Proposition \ref{theo_loc}.\quad
  Assume $\tm$ and $(v,u,\Theta)$ to be as provided by Lemma \ref{lem_loc} when applied to
  $(v_0,u_0,\Theta_0):=(u_{0t}+au_0,u_0,\Theta_0)$. Then we may use (\ref{l1}), (\ref{0v}) and (\ref{ext})
  to verify (\ref{tl1}), (\ref{tl2}), (\ref{0}) and (\ref{Ext}).
\qed

 \color{black}
\mysection{Capitalizing on sublinear temperature dependencies to establish large existence times. Proof of Theorem 1.2}
From here on forward, we denote the classical solution of $\eqref{0v}$ corresponding to initial data $(u_0,u_{0t},\Theta_0)$ fulfilling $\eqref{Init}$ and $\eqref{12.2}$ as $(u,v,\Theta)$.
Bearing in mind condition \eqref{ext}, we denote some fundamental properties of the gradients of $u$, $v$ and $\Theta$, which arise from straightforward testing procedures. Subsequently, we will use our strong assumptions on $f$ and $\gamma$ – specifically their sublinear temperature dependence – to regulate the undesirable terms that emerge during the process.

Firstly, we note down some basic evolution features for the gradients of the mechanical components $u$ and $v$ resulting from testing the first two lines of the system \eqref{0v}.

\begin{lem}\label{lem3}
Assume $a,D,\gamma$ and $f$ to satisfy \eqref{Vorr} and \eqref{Vorr2}, then
 
  \be{4.1}
  \frac{1}{2} \frac{d}{dt} \io u_x^2
  + \frac{a}{2} \io  u_x^2
  \le \frac{1}{2a}\io v_x^2
  \qquad \mbox{for all } t\in (0,\tm)
  \ee
  and
  \be{4.1b}
  \frac{1}{4} \frac{d}{dt} \io u_x^{4}
  + \frac{a}{2} \io u_x^{4}
  \le \frac{8}{a^3} \io v_x^{4}
  \qquad \mbox{for all } t\in (0,\tm)
  \ee
as well as
  \bea{3.1}
	\hs{-8mm}
	\frac{1}{2}\frac{d}{dt} \io  v_x^2
	+ \frac 12\io \gamma(\Theta) v_{xx}^2 \le \io \frac{\gamma'(\Theta)^2}{\gamma(\Theta)} v_x^2\Theta_x^2 +2a\io v_x^2 + \frac { a^3}4 \io u_x^2 +\io \frac{f'(\Theta)^2}{\gamma(\Theta)} \Theta_x^2
  \eea
 for all $t\in(0,\tm)$.
\end{lem}

\color{black}

\proof
By testing the second equation from \eqref{0v}, we obtain
\bea{22}
\frac{1}{2} \frac{d}{dt} \io u_x^2 =  \io u_x u_{xt}&=&\io u_x v_x -a\io u_x^2 \nn
  \\
  &\le& \frac1 {2a}\io v_x^2-\frac a 2 \io u_x^2
    \qquad \mbox{for all } t\in (0,\tm).
  \eea
In quite a similar fashion, \eqref{4.1b} can be accomplished. By testing the first equation in conjunctiopn with the boundary condition $v_x\big|_{\pO\times(0,\tm)}=0$ and Young's inequality, we infer that
\bea{22}
\frac{1}{2}\frac{d}{dt} \io  v_x^2 &=& \io v_x v_{xt}\nn\\
&=& \io (\gamma(\Theta) v_x)_{xx}v_x+a\io v_x^2 -a^2\io u_x v_x+\io (f(\Theta))_{xx} v_x\nn\\
&\le& - \io \gamma(\Theta) v_{xx}^2-\io \gamma'(\Theta) v_x \Theta_x v_{xx}+2a\io v_x^2 + \frac { a^3}4 \io u_x^2 -\io f'(\Theta)\Theta_xv_{xx} \nn\\
&\le& -\io \gamma(\Theta)v_{xx}^2 +\frac 12\io \gamma(\Theta) v_{xx}^2 +\io \frac{\gamma'(\Theta)^2}{\gamma(\Theta)} v_x^2\Theta_x^2 +2a\io v_x^2 + \frac { a^3}4 \io u_x^2+\io \frac{f'(\Theta)^2}{\gamma(\Theta)} \Theta_x^2\nn
\eea
for all $t\in(0,\tm)$. \qed
\color{blue}

\color{black}
For the temperature $\Theta$, we denote a similar statement, which also results from simple testing procedures. 
\begin{lem}\label{lem5}

Assume  $a,D$ and $F$ to satisfy \eqref{Vorr}. Then for every $0<\ugamma<\ogamma$ and any function  $\Gamma$ fulfilling \eqref{Vorr} and \eqref{gamBesch}, one can find $K_0=K_0(a,D,\ogamma)>0$ such that 
  \bea{5.1}
	\frac{1}{2} \frac{d}{dt} \io \Theta_x^2
	+ \frac D2 \io \Theta_{xx}^2
	&\le& K_0 \io   v_x^4+ K_0 \io  u_x^4+\io F(\Theta)^2 u_{xt}^2 
  \eea
 for all $ t\in (0,\tm)$.
\end{lem}
\proof
 Since  $\Theta(\cdot,t)|_{\pO}=0$ for all $t\in(0,\tm)$, we infer $\Theta_t=0$ on $\pO$. Furthermore, by applying Young's inequality we obtain
\bea{1t}
\frac 1 2\frac{d}{dt} \io \Theta_x^2 &=& \io \Theta_x \Theta_{xt}\nn\\
&=& \io \Theta_x\cdot\Big(D\Theta_{xx}+\Gamma(\Theta)|v_x-au_x|^2+F(\Theta)u_{xt}\Big)_x\nn\\
&=& -D\io \Theta_{xx}^2-\io \Theta_{xx} \Big( \Gamma(\Theta)|v_x-au_x|^2+F(\Theta)u_{xt} \Big)\nn\\
& &+\Big[\Theta_x \Big(D\Theta_{xx}+\Gamma(\Theta)|v_x-au_x|^2+F(\Theta)u_{xt}\Big)\Big]_{\pO}\nn\\
&\le& -\frac D 2 \io \Theta_{xx}^2 + \frac {\ogamma^2} {D}\io (v_x^4 +a^4 u_x^4) +\io F(\Theta)^2 u_{xt}^2\nn
\eea
for all $t\in(0,\tm).$ The claim follows for $K_0:=\frac{\ogamma^2}D (1+a^4)$.

We now finalize our preparation for the main comparison argument by taking advantage of the fact that we can control the coefficient of the cubic term in the underlying ODE inequality for our energy functionals. Furthermore, it becomes evident that the restrictive assumption $F(s)\le C_F(1+s)^\alpha$, for some $C_F>0$ and $\alpha\in(0,1)$, as indicated in \eqref{Vorr}, is indispensable.

\begin{lem}\label{lem8}
 Assume that $a,D$ and $F$ satisfy \eqref{Vorr} for positive constants $C_F$ and $\alpha\in(0,1)$. Then, for every $0<\ugamma<\ogamma$, there exist constants $k_1=k_1(a,D,\ugamma)$ and $K_1=K_1(\Om,a,D,\ogamma,\ugamma,C_F,\alpha)$ such that, for any $\delta_\star\in(0,1)$ and all functions $\gamma$, $\Gamma$ and $f$, which not only fulfill \eqref{Vorr2}, but also \eqref{gamBesch} as well as \eqref{12.3}, the inequalities
 \bea{7.3}
	& & \hs{-20mm}
	\frac{d}{dt} \bigg\{ \io v_x^2 + 4a^2 \io u_x^2 \bigg\}
	+ k_1\io  v_{xx}^2
	+ k_1 \io u_x^2 \nn\\
	&\le& \delta_\star^\frac{7}{2}K_1\io \Theta_x^4
	+ \delta_\star^\frac{1}{2} K_1 \io v_x^4+ 8a \io v_x^2  + \delta_\star^2 K_1 \io  \Theta_x^2,
  \eea
 and 
 \bea{8.1}	\hs{-10mm}
	\frac{d}{dt} \bigg\{\ds\io \Theta_x^2 + \dse\rho\io u_x^{4} \bigg\}
	+  \ds k_1 \io \Theta_{xx}^2
	+ \dse k_1  \io u_x^{4} 
	&\le& \dse K_1 \io v_x^4 + \delta_\star^{\frac{1-\beta}\alpha} K_1\io \Theta_x^4 +K_1
		  \eea
  hold for all $t\in(0,\tm)$, where
  $\beta\equiv\beta(\alpha):=\frac{1-\alpha}{2},\ \rho\equiv \rho(a,b,D,\ogamma):=\frac{K_0\ogamma^2+1}a$. 
\end{lem}

\proof
In view of  Lemma $\ref{lem3}$ and Lemma $\ref{lem5}$,  we obtain the following by multiplying with $\ds$ and applying Young's inequality once again
\bea{5.1}
	 \frac{d}{dt} \ds \io \Theta_x^2
	+ \ds  \frac D2 \io \Theta_{xx}^2
	&\le& \ds 2K_0\ogamma^2 \io   v_x^4+ \ds 2K_0 \ogamma^2 \io  u_x^4 + \ds 2\io F(\Theta)^2  \Big(v_x^2+a^2u_x^2\Big)\nn 
\\
&\le&  2\big(\ds K_0\ogamma^2+\dse\big) \Big\{\io   v_x^4+ \io  u_x^4 \Big\}+\delta_\star^{1-\beta} C_F^4  \io (\Theta+1)^{4\alpha} \nn\\
&\le&  2\dse\big( K_0\ogamma^2+1\big) \Big\{ \io   v_x^4+ \io  u_x^4\Big\}+ \delta_\star^{\frac{1-\beta}{\alpha}} C_F^4\io (\Theta+1)^4+C_F^4|\Om|\nn
\eea
for all $t\in(0,\tm)$. With an application of a Poincaré inequality, which states that we are able to find a constant $C_p(\Om)>0$ such that
$$\io \vp ^4 \le C_p \io \vp_x^4 \qquad \mbox{for all }\vp\in W^{1,4}_0(\Om),$$
on $\vp(x)=\Theta(x,t)$ for every $t>0$, we estimate
$$\delta_\star^{\frac{1-\beta}\alpha}C_F^4 \io (\Theta+1)^4\le
\delta_\star^{\frac{1-\beta}\alpha}8C_F^4 \io (\Theta^4+1)\le 
\delta_\star^{\frac{1-\beta}\alpha}8C_F^4 C_p \io \Theta_x^4 +8C_F^4|\Om|\quad \mbox{ for all } t\in(0,\tm),$$
where we used $\ds<1$.
 From Lemma 3.1, we conclude not only
 \be{4.2}
  \frac{d}{dt} \dse \rho\io u_x^{4}
  + \dse 2a\rho \io u_x^{4}
  \le \dse \rho\frac{32}{a^3} \io v_x^{4}
  \qquad \mbox{for all } t\in (0,\tm),
  \ee
but also, taking \eqref{12.3} into account,
\bea{4.1}
   \frac{d}{dt}\Big\{ 4a^2 \io u_x^2
  +\io  v_x^2\Big\}&+&4a^3\io u_x^2  + \io \gamma(\Theta) v_{xx}^2\nn\\
  &\le&\io \frac{\gamma'(\Theta)^2}{\gamma(\Theta)} v_x^2\Theta_x^2 +8a\io v_x^2 + \frac { a^3}2 \io u_x^2 +\io \frac{f'(\Theta)^2}{\gamma(\Theta)} \Theta_x^2\nn\\
&\le& \frac{\delta_\star^{\frac 1 2}}{2\ugamma}\io v_x^4 +\frac{\delta_\star^{\frac 7 2}}{2\ugamma} \io \Theta_x^4 +8a\io v_x^2+\frac{a^3}2\io u_x^2 + \frac{\delta_\star^2}{\ugamma} \io \Theta_x^2 \nn
    \eea
for all $t\in(0,\tm)$.\\
For $k_1:=\min\{\frac D2, \  \ugamma, \ 2a\rho, \ 4a^3\}$ and $K_1:=\max\Big\{2(K_0\ogamma^2+1)+\rho\frac{32}{a^3}, \ \frac 1 {\ugamma}, \ 8C_F^4C_p, \ 9C_F^4|\Om|)\Big\}$ the claim follows.
\qed

The fourth-order gradients that appear on the right-hand side of \eqref{7.3} and \eqref{8.1} pose a challenge that forces us to retreat to proving solutions with arbitrary but finite existence times, rather than attempting to prove global existence. To control these terms, we estimate them against the naturally occurring negatively signed diffusion terms on the one hand and the cubic terms on the other. In the latter case, we can control the coefficients that occur using the sublinear temperature dependencies. Here, we do this for general $\delta\in(0,1)$.

\begin{lem}
 Assume $a,D$ and $F$ to satisfy \eqref{Vorr} for some positive constants $C_F$ and $\alpha\in(0,1)$. For any $\delta\in(0,1)$, $\eps>0$ and $P>0$ there exists $Q_\eps=Q_\eps(\Om,P,a,\eps)>0$ such that
\begin{align}\label{k2}
\delta^{\frac 72}P\io \Theta_x^4&+ \delta^{\frac{1-\beta}\alpha}P\io \Theta_x^4+ \delta^{\beta}P\io v_x^4+\delta^{\frac 12}P\io v_x^4 \nn\\  
&\le \delta^{\frac 12}\frac\eps2\io \Theta_{xx}^2 +\delta^{\frac 12} a \io \Theta_x^2 +\frac\eps2\io v_{xx}^2 + a  \io v_x^2+ \delta^{\kappa} Q_\eps\Big(\delta^{\frac 12}\io \Theta_x^2+\io v_x^2\Big)^3
\end{align}

for all $t\in(0,\tm)$, where $\beta:=\frac{1-\alpha}2$ and $\kappa\equiv\kappa(\alpha):=\min\big\{\beta,\frac 2\alpha(1-\beta-\alpha)\big\}$. 
\end{lem}
\proof

The application of a Gagliardo-Nirenberg inequality enables the determination of $C_{GN}>0$ such that
\bea{GNU}
\big \| \vp\big\|_{L^4(\Om)}^4 
&\le& C_{GN}\big\| \vp_{x}\big\|_{L^2(\Om)}\big\|\vp\big\|_{L^2(\Om)}^3+C_{GN}\big\|\vp\big\|_{L^2(\Om)}^4\qquad \mbox{for all }\vp\in C^2(\Om).\nn
\eea
By setting $\vp=\Theta_x$ in the aforementioned inequality, a combination with Young's inequality yields  for $\hat P= 2P\cdot C_{GN} $
\bea{neu2}
2P \io \Theta_x^4 &=& 2P \| \Theta_x\|_{L^4(\Om)}^4 \nn\\
&\le& \hat P\| \Theta_{xx}\|_{L^2(\Om)}\|\Theta_x\|_{L^2(\Om)}^3+\hat P\|\Theta_x\|_{L^2(\Om)}^4\nn\\
&=& \hat P\Big(\io \Theta_{xx}^2\Big)^{\frac 12} \cdot \Big(\io \Theta_x^2\Big)^{\frac 32}+\hat P\Big(\io \Theta_x^2 \Big)^2\nn\\
&\le& \frac{\eps}2\io \Theta_{xx}^2+ \frac{\hat P^2}{2\eps} \Big(\io\Theta_x^2\Big)^3 +\hat P\Big(\io \Theta_x^2\Big)^2 \qquad\mbox{for all } t\in(0,\tm).
\eea
Utilizing Young's inequality once again, we note
\bea{neu5}
\delta^{\frac 72} \hat P\Big(\io \Theta_x^2\Big)^2 &=& \Big\{\delta^{\frac 12}\io \Theta_x^2\Big\}\cdot \Big\{ \sqrt{a}\cdot \delta^{3}\frac{\hat P}{\sqrt{a}}  \cdot \Big(\io \Theta_x^2\Big) \Big\}\nn\\
&\le& \Big\{\delta^{\frac 12}  \io \Theta_x^2\Big\}\cdot  \Big\{a+\delta^{6}\frac{\hat P^2}a \Big(\io \Theta_x^2\Big)^2\Big\}\nn\\
&\le& \delta^{\frac 12}  a\io \Theta_x^2+ \delta^{5} \frac{\hat P^2}a\Big(\delta^{\frac 12}\io \Theta_x^2\Big)^3 \qquad\mbox{ for all }t\in(0,\tm),\nn
\eea
and similarly
\bea{neu5}
\delta^{\frac{1-\beta}\alpha} \hat P\left(\io \Theta_x^2\right)^2 &=& \left\{\delta^{\frac12}\io \Theta_x^2\right\}\cdot \left\{ \sqrt{a}\cdot \delta^{\frac {1-\beta}\alpha-\frac 12}\frac{\hat P}{\sqrt{a}}  \cdot \left(\io \Theta_x^2\right) \right\}\nn\\
&\le& \left\{\delta^{\frac12}  \io \Theta_x^2\right\}\cdot  \left\{a+ \delta^{\frac {2-2\beta-\alpha}\alpha}\frac{\hat P^2}a \left(\io \Theta_x^2\right)^2\right\}\nn\\
&\le& \delta^{\frac12}  a\io \Theta_x^2+  \delta^{\frac2\alpha(1-\beta-\alpha)}\frac{\hat P^2}a\left(\delta^{\frac12}\io \Theta_x^2\right)^3 \quad\mbox{ for all }t\in(0,\tm),\nn
\eea
where in view of the definition of $\beta:=\frac {1-\alpha}2 $, the positivity of

\bea{pos}
 1-\beta-\alpha&=& 1-\frac 12+\frac\alpha 2-\alpha\nn\\
 &=& \frac12(1-\alpha)\nn
 \eea
\color{black}
is guaranteed due to $\alpha\in(0,1)$. Now, we multiply \eqref{neu2} by $\delta^m$, with $m:=\max\big\{\frac 72,\frac{1-\beta}\alpha\big\}>\frac 12$ and since $0<\delta<1$ we infer

\bea{neu21}
 \delta^{\frac72}P\io \Theta_x^4 +\delta^{\frac{1-\beta}\alpha} P \io \Theta_x^4 &\le& \delta^m 2P\io\Theta_x^4 
\nn\\
&\le& \delta^{\frac 12}\frac{\eps}2\io \Theta_{xx}^2 +\delta^{\frac 12} a \io \Theta_x^2+ \delta^{\kappa} \left(\frac{\hat P^2}a+\frac{\hat P^2}{2\eps}\right)\left(\delta^{\frac 12}\io \Theta_x^2\right)^3 \nn
\eea

for all $t\in(0,\tm)$. In similar fashion we are able to estimate
\bea{neu3}
2P\io v_x^4 \le \frac{\eps} 2 \io v_{xx}^2+\frac{\hat P^2}{2\eps}\left(\io v_x^2\right)^3+\hat P\left(\io v_x^2\right)^2\qquad\mbox{for all } t\in(0,\tm),
\nn\eea

as well as

\bea{neu6}
\hat P\left(\io v_x^2\right)^2 
\le a \io v_x^2+ \frac{\hat P^2}{a}\left(\io v_x^2\right)^3\qquad\mbox{ for all }t\in(0,\tm),\nn
\eea
and since $\delta^{\frac 12}<\delta^\beta$ applies, we conclude from comining
\bea{neu32}
\delta^{\frac12}P\io v_x^4 +\delta^{\beta}P\io v_x^4 \le 2P\delta^{\kappa}\io v_x^4 \le \frac{\eps} 2 \io v_{xx}^2+a\io v_x^2+\delta^{\kappa}\left(\frac{\hat P}{2\eps}+\frac{\hat P^2}a \
\right)\left(\io v_x^2\right)^3 \nn
\eea
for all $t\in(0,\tm).$
Setting $Q_\eps:=\frac{\hat P^2}{2\eps}+\frac{\hat P^2}a$, we may readily infer \eqref{k2}.\\

\qed

With our preparations complete, we are now in a position to establish the final comparative argument. The temporal development of the relevant energy functional can be traced by using the preceding lemmas, the sublinear temperature dependencies as formulated in \eqref{12.3}, by means of a cubic ODE, which explodes after a finite time, but only after any prescribed time $T_\star$.

\begin{lem}\label{lem9}
Suppose $a,D,$ and $F$ are such that \eqref{Vorr} is satisfied for some positive constants $C_F$ and $\alpha\in(0,1)$ and that $(u_0, u_{0t},\Theta_0)$ are such that \eqref{Init} holds.
  Then for every $M$, $\ugamma$, $\ogamma$, and $T_\star>0$, there exists a constant $\delta_\star=\delta_\star(M,T_\star,a,D,\Om,\ugamma,\ogamma,C_F,\alpha)>0$ such that for all functions $u_0, u_{0t}, \Theta_0, \gamma,\Gamma$ and $f$ which satisfy $\eqref{Vorr2}$, $\eqref{gamBesch}$ and $\eqref{12.3}$, it follows that in Proposition \ref{theo_loc} actually $\tm\ge T_\star$ holds.

\end{lem}

\proof

By linear superposition, Lemma 3.3 leads to 
 \begin{align*}
	\frac{d}{dt} &\bigg\{ \delta_\star^{\frac 12}\io \Theta_x^2 + \delta_\star^{\beta}\rho \io u_x^{4} +\io v_x^2 + 4a^2 \io u_x^2 \bigg\}
	+ \delta_\star^{\frac 12} k_1 \io \Theta_{xx}^2
	+ \delta_\star^{\beta} k_1 \io u_x^{4}+k_1\io v_{xx}^2+k_1\io u_x^2 \\
	&\le \delta_\star^{\beta}K_1\io v_x^4+\delta_\star^{\frac{1-\beta}\alpha} K_1\io \Theta_x^4+ \delta_\star^\frac{7}{2}K_1\io \Theta_x^4
	+ \delta_\star^\frac{1}{2} K_1\io v_x^4+ 8a \io v_x^2 +\delta_\star^{\frac 12} K_1 \io \Theta_x^2 +K_1 \nn
  \end{align*}
for all $t\in(0,\tm)$.

Applying Lemma 3.4 to $\eps=k_1(a,D,\ugamma)$ and $P=K_1(\Om,a,D,\ugamma,\ogamma,C_F,\alpha)$, we are able to estimate for $K_2=K_2(\Om,a,D,\ugamma,\ogamma,C_F,\al):=Q_\eps(\Om,a,D,\ugamma,\ogamma,C_F,\al)>0$
\begin{align}\label{c21}	
	\frac{d}{dt} \bigg\{ \delta_\star^{\frac 12}\io \Theta_x^2 + \delta_\star^{\beta}\rho \io u_x^{4} &+\io v_x^2 + 4a^2 \io u_x^2 \bigg\}
	+ \delta_\star^{\frac 12} \frac{k_1}2 \io \Theta_{xx}^2
	+ \delta_\star^{\frac 12} k_1 \io u_x^{4} +\frac{k_1}2\io v_{xx}^2+k_1\io u_x^2 \nn\\
	&\le  \delta_\star^{\kappa} K_2\Big(\delta_\star^{\frac 12}\io \Theta_x^2+\io v_x^2\Big)^3+ 9a \io v_x^2 +\delta_\star^{\frac 12} (a+K_1) \io \Theta_x^2 +K_1
  \end{align}
for all $t\in(0,\tm)$.

We now fix the positive constants 
\bea{c20}
\hs{-15mm}& & \chi:=\max\big\{1,4a^2,\rho\big\},\quad  \tau:=\max\big\{9a,a+K_1\big\},\quad s_0:=\max\left\{0,\frac 1 \tau\ln\left(\frac {\chi M+1} {\sqrt\tau}\right)\right\}\nn\\
 \mbox{and }\hs{-12mm} & & \nn\\
 \hs{-15mm} & &\delta_\star=\delta_\star(\Om,a,D,\ugamma,\ogamma,M,T_\star,C_F,\alpha):=\min\left\{1,\left(\frac{e^{-2\tau(s_0+T_\star)}}{K_2}\right)^{\frac 1\kappa}
\right\},\qquad \qquad \nn\eea
and define
$$y(t):= 1+ \delta_\star^{\frac 12} \io \Theta_x^2+\delta_\star^{\beta} \rho \io u_x^4 +\io v_x^2 + 4a^2 \io u_x^2\qquad\mbox{ for every } t\in[0,\tm).$$
Due to \eqref{c21}, we are able to conclude
\bea{neu6} 
y'(t)\le \delta_\star^{\kappa} K_2 y(t)^3+\tau y(t) \qquad\mbox{ for all } t\in(0,\tm).
\eea
In preparation for a comparison argument, we define a function
 \bea{OL}
 \hat y(t):= \frac{\sqrt\tau e^{\tau(s_0+t)}}{\sqrt{2-\sigma e^{2\tau(s_0+t)}}}\qquad \mbox{ on } [0,T_\star],
 \eea
 where $\sigma:=\delta_\star^{\kappa}K_2$, which grows
  monotonously and is well-defined on $[0,T_\star]$ since our previously selected constants guarantee
$$ \sigma e^{2\tau(s_0+T_\star)}\le K_2\Bigg(\frac{e^{-2\tau(s_0+T_\star)}}{K_2}\Bigg) \cdot e^{2\tau(s_0+T_\star)}=1.$$
For the derivative we calculate
\bea{neu7}
\hat y'(t)&=& \tau  \frac{\sqrt\tau e^{\tau(s_0+t)}}{\sqrt{2-\sigma e^{2\tau(s_0+t)}}}+ \frac{\sqrt\tau e^{\tau(s_0+t)}}{2\Big(2-\sigma e^{2\tau(s_0+t)}\Big)^{\frac 3 2}}\cdot e^{2\tau(s_0+t)} \cdot2\tau\sigma\nn\\
&=& \tau \hat y(t)+\sigma \hat y(t)^3 \quad\qquad \mbox{ for all } t\in(0,T_\star) \nn
\eea
and for the initial condition we are able to denote that if $\frac{\chi M+1}{\sqrt{\tau}}\ge 1$
\bea{neu8}
\hat y(0)=\frac{\sqrt\tau e^{\tau s_0}}{\sqrt{2-\sigma e^{2\tau s_0}}}\ge \sqrt \tau e^{\ln(\frac {\chi M+1}{\sqrt \tau})} \ge \chi M+1 \nn
\eea
holds, and if $\frac{\chi M+1}{\sqrt\tau}<1$
\bea{neu8b}
\hat y(0)\ge \sqrt{\tau}>\chi M+1,\nn
\eea
so in total we obtain $\hat y(0)\ge \chi M+1 \ge y(0)$. We now proceed on the assumption that $\tm<T_\star$, so we would be able to conclude from a comparison argument that \bea{yle}
y(t)\le \hat y(t)\le \hat y(T_\star)=\frac{\sqrt\tau e^{\tau(s_0+T_\star)}}{\sqrt{2-\sigma e^{2\tau(s_0+T_\star)}}} < \sqrt\tau e^{\tau(s_0+T_\star)}  \mbox{ on }(0,\tm).
\eea 
We note that according to known smoothing properties of the Dirichlet heat semi group $(e^{t\Delta})_{t\ge0}$ on $\Om$ there exists $C_1=C_1(D)>0$  such that 
\bea{neu9}
\|e^{tD\Delta} \phi\|_{L^\infty(\Om)} &\le& C_1(1+t^{-\frac 12} )\|\phi\|_{L^1(\Om)}\qquad \mbox{ for all } t\in(0,T) \mbox{ and } \phi\in C^2(\bom).
\eea
For the Duhamel representation associated with the third equation in \eqref{0v}, and by the maximum principle and \eqref{neu9}, our assumption would lead to

\bea{neu10}
\|\Theta(\cdot,t)\|_{L^\infty(\Om)}
	&=& \bigg\| e^{tD\Del} \Theta_0 + \int_0^t e^{(t-s)D\Del} 
		\Big\{ \Gamma(\Theta(\cdot,s)) \big| v_x(\cdot,s)-a u_x(\cdot,s)\big|^2\Big\}ds\nn\\
        & &\qquad \qquad +\int_0^t e^{(t-s)D\Delta}\Big\{F(\Theta(\cdot,s)) (v_x(\cdot,s)-au_x(\cdot,s))\Big\} ds \bigg\|_{L^\infty(\Om)} \nn\\
	&\le& \big\| e^{tD\Del} \Theta_0\big\|_{L^\infty(\Om)}
	+ \int_0^t \Big\| e^{(t-s)D\Del} 
		\Big\{ \Gamma(\Theta(\cdot,s)) \big| v_x(\cdot,s)-a u_x(\cdot,s)\big|^2 ds \Big\|_{L^\infty(\Om)} \nn\\
         & &\qquad \qquad +\int_0^t\Big\| e^{(t-s)D\Delta}\Big\{F(\Theta(\cdot,s)) (v_x(\cdot,s)-au_x(\cdot,s))\Big\} ds \Big\|_{L^\infty(\Om)} \nn\\
	&\le& \big\|\Theta_0\big\|_{L^\infty(\Om)}
	+C_1 \int_0^t \big(1-(t-s)^{-\frac 12}\big)
		\Big\| \Gamma(\Theta(\cdot,s)) \big| v_x(\cdot,s)-a u_x(\cdot,s)\big|^2 \Big\|_{L^1(\Om)} ds \nn\\
        & &\qquad \qquad   +C_1\int_0^t (1-(t-s)^{-\frac 1 2}) \Big\| F(\Theta(\cdot,s)) (v_x(\cdot,s)+au_x(\cdot,s))\Big\|_{L^1(\Om)} ds
\eea
for every $t\in(0,\tm)$. Furthermore, we would be able to estimate
\bea{L1}
& &C_1\int_0^t \big(1-(t-s)^{-\frac 12}\big)
		\Big\| \Gamma(\Theta(\cdot,s)) \big| v_x(\cdot,s)-a u_x(\cdot,s)\big|^2 \Big\|_{L^1(\Om)} ds \nn\\
        &\le& 2C_1  \ogamma\int_0^t (1-(t-s)^{-\frac 12})\io\big(v_x(x,s)^2+a^2u_x(x,s)^2\big)dx\ ds\nn\\
        &\le& 8C_1  \ogamma  t y(t)\nn\\
        &\le& 8C_1 \ogamma T_\star \hat y(T_\star)<\infty \qquad\mbox{for all } t\in(0,\tm),
        \eea
as well as
\bea{L2}
& &  C_1\int_0^t (1-(t-s)^{-\frac 1 2}) \Big\| F(\Theta(\cdot,s)) (v_x(\cdot,s)+au_x(\cdot,s))\Big\|_{L^1(\Om)} ds\nn\\
&\le& C_1\int_0^t(1-(t-s)^{\frac12})\left(\io F(\Theta)^2 +\io v_x^2+a^2\io u_x^2\right) \ ds\nn\\
 &\le&  \delta_\star^{-\frac 12}\hat y(t) C_1(2C_pC_F^2+2C_F^2|\Om|+4) \nn\\
 &\le& \delta_\star^{-\frac 12}\hat y(T_\star)C_1(2C_pC_F^2+2C_F^2|\Om|+4)<\infty\qquad \mbox{ for all } t\in(0,\tm).
        \eea

By combining \eqref{L1} and \eqref{L2} with \eqref{neu10}, we would be able to conclude that
\bea{Grenze}
\|\Theta(\cdot,t)\|_{L^\infty(\Om} \le \|\Theta_0\|_{L^\infty(\Om)} + 2C_1 \ogamma (1+a^2)T_\star \hat y(T_\star)+\delta_\star^{-\frac 12}\hat y(T_\star)C_1(2C_pC_F^2+2C_F^2|\Om|+4)<\infty\nn
\eea

for all $t\in(0,\tm)$ and further
\bea{c30} 
& & \hs{-15mm}
	\limsup_{t\nearrow\tm} \Big\{  \|\Theta(\cdot,t)\|_{L^\infty(\Om)} \Big\}
		<\infty. \nn	
\eea
Furthermore we immediately obtain from \eqref{0v} 
\bas
	\frac{d}{dt} \io v - a\frac{d}{dt} \io u=0
	\qquad \mbox{for all } t\in (0,\tm),
  \eas
which directly yields

\be{mass}
	\io \big(v(\cdot,t)-au(\cdot,t)\big)=\io (v_0-au_0)
	\qquad \mbox{for all } t\in (0,\tm).\nn
  \ee
So 
\be{m1}
	\Big\|v(\cdot,t) -a u(\cdot,t)\Big\|_{W^{1,2}(\Om)} 
	\le C_2 \Big\| v_x(\cdot,t)-a  u_x(\cdot,t)\Big\|_{L^2(\Om)} 
	+ C_2 \cdot \bigg| \io (v_0-au_0)\bigg|\nn
  \ee
results for all $t\in(0,\tm)$ and some $C_2>0$ by means of a Poincaré-type inequality, which states that there exists $C_3>0$ such that
$$\left\| \vp - \frac{1}{|\Om|}\io \vp \right\|_{W^{1,2}(\Om)} \le C_3\| \vp_x\|_{L^2(\Om)}$$
for all $\vp\in W^{1,2}(\Om)$.  Now, recalling \eqref{yle}, we conclude that there would exist a constant $C_4>0$ such that
\be{ext2}
\|v(\cdot,t) -a u(\cdot,t)\|_{W^{1,2}(\Om)} \le  C_4 \hat y(t)\le C_4\sqrt\tau e^{\tau(s_0+T_\star)} \qquad\mbox{ for all } t\in(0,\tm)\nn,
\ee
which would imply 
\be{ext3} 
\limsup_{t\nearrow\tm} \Big\{  \|v(\cdot,t) -a u(\cdot,t)\|_{W^{1,2}(\Om)} \Big\}
		<\infty.	\ee
		The combination of the statements \eqref{c30} and \eqref{ext3} leads to a contradiction with the extensibility criterion \eqref{ext}.
		 \qed

\textbf{Data availability statement.}
Data sharing is not appicable to this article since no datasets were generated or analyzed during the current study.

\textbf{Acknowledgment.}
The author acknowledges the support provided by the Deutsche Forschungsgemeinschaft
(Project No. 444955436) and further declares that he has no conflict of interest.

\printbibliography

@Article{Blanchard2000,
  author     = {Blanchard, D. and Guib\'e, O.},
  journal    = {Adv. Differential Equations},
  title      = {Existence of a solution for a nonlinear system in thermoviscoelasticity},
  year       = {2000},
  issn       = {1079-9389},
  number     = {10-12},
  pages      = {1221--1252},
  volume     = {5},
  fjournal   = {Advances in Differential Equations},
  mrclass    = {74H20 (35Q72 74D10 74F05)},
  mrnumber   = {1785674},
  mrreviewer = {Song\ Mu\ Zheng},
doi = {10.57262/ade/1356651222}
}

@Article{Racke_1997,
  author    = {Racke, Reinhard and Zheng, Songmu},
  journal   = {Journal of Differential Equations},
  title     = {Global Existence and Asymptotic Behavior in Nonlinear Thermoviscoelasticity},
  year      = {1997},
  issn      = {0022-0396},
  month     = feb,
  number    = {1},
  pages     = {46–67},
  volume    = {134},
  doi       = {10.1006/jdeq.1996.3216},
  publisher = {Elsevier BV},
  url       = {http://dx.doi.org/10.1006/jdeq.1996.3216},
}

@Article{Dafermos_1982,
  author    = {Dafermos, C.M. and Hsiao, L.},
  journal   = {Nonlinear Analysis: Theory, Methods $\&$ Applications},
  title     = {Global smooth thermomechanical processes in one-dimensional nonlinear thermoviscoelasticity},
  year      = {1982},
  issn      = {0362-546X},
  month     = may,
  number    = {5},
  pages     = {435–454},
  volume    = {6},
  doi       = {10.1016/0362-546x(82)90058-x},
  publisher = {Elsevier BV},
  url       = {http://dx.doi.org/10.1016/0362-546x(82)90058-x},
}

@Article{Jiang_1993,
  author    = {Jiang, Song},
  journal   = {Quarterly of Applied Mathematics},
  title     = {Global large solutions to initial-boundary value problems in one-dimensional nonlinear thermoviscoelasticity},
  year      = {1993},
  issn      = {1552-4485},
  month     = dec,
  number    = {4},
  pages     = {731–744},
  volume    = {51},
  doi       = {10.1090/qam/1247437},
  publisher = {American Mathematical Society (AMS)},
  url       = {http://dx.doi.org/10.1090/qam/1247437},
}

@Article{Kawohl_1985,
  author    = {Kawohl, Bernhard},
  journal   = {Journal of Differential Equations},
  title     = {Global existence of large solutions to initial boundary value problems for a viscous, heat-conducting, one-dimensional real gas},
  year      = {1985},
  issn      = {0022-0396},
  month     = jun,
  number    = {1},
  pages     = {76–103},
  volume    = {58},
  doi       = {10.1016/0022-0396(85)90023-3},
  publisher = {Elsevier BV},
  url       = {http://dx.doi.org/10.1016/0022-0396(85)90023-3},
}

@Article{Slemrod_1981,
  author    = {Slemrod, M.},
  journal   = {Archive for Rational Mechanics and Analysis},
  title     = {Global existence, uniqueness, and asymptotic stability of classical smooth solutions in one-dimensional non-linear thermoelasticity},
  year      = {1981},
  issn      = {1432-0673},
  month     = jun,
  number    = {2},
  pages     = {97–133},
  volume    = {76},
  doi       = {10.1007/bf00251248},
  publisher = {Springer Science and Business Media LLC},
  url       = {http://dx.doi.org/10.1007/bf00251248},
}

@Article{Jiang_1990,
  author    = {Jiang, Song},
  journal   = {Proceedings of the Royal Society of Edinburgh: Section A Mathematics},
  title     = {Global existence of smooth solutions in one-dimensional nonlinear thermoelasticity},
  year      = {1990},
  issn      = {1473-7124},
  number    = {3–4},
  pages     = {257–274},
  volume    = {115},
  doi       = {10.1017/s0308210500020631},
  publisher = {Cambridge University Press (CUP)},
  url       = {http://dx.doi.org/10.1017/s0308210500020631},
}

@Book{GutierrezLemini2013,
  author    = {Gutierrez-Lemini, D.},
  publisher = {Springer US},
  title     = {Engineering Viscoelasticity},
  year      = {2013},
  isbn      = {9781461481393},
  lccn      = {2013942016},
  url       = {https://books.google.de/books?id=MbjBAAAAQBAJ},
  doi = {10.1007/978-1-4614-8139-3}
}

@Book{Chawla2008,
  author    = {Chawla, Krishan Kumar and Meyers, Marc André},
  publisher = {Cambridge University Press},
  title     = {Mechanical Behavior of Materials, Third Edition},
  year      = {2008},
  isbn      = {9781107386358},
  url       = {https://books.google.de/books?id=8UYwnQAACAAJ},
doi = {10.1017/CBO9780511810947}
}

@Book{Boley1960,
  author    = {Boley, Bruno A. and Weiner, Jerome H.},
  publisher = {Wiley},
  title     = {Theory of Thermal Stresses},
  year      = {1960},
  url       = {https://books.google.de/books?id=HW5vyQEACAAJ},
}

@Article{Claes2024,
  author   = {Claes, Leander and Lankeit, Johannes and Winkler, Michael},
  journal  = {Mathematical Models and Methods in Applied Sciences},
  title    = {A model for heat generation by acoustic waves in piezoelectric materials: Global large-data solutions},
  year     = {2025},
  number   = {11},
  pages    = {2465-2512},
  volume   = {35},
  doi      = {10.1142/S0218202525500447},
  eprint   = { https://doi.org/10.1142/S0218202525500447
},
  url      ={https://doi.org/10.1142/S0218202525500447},
}

@Article{Roubicek2009,
  author    = {Roubíček, Tomáš},
  journal   = {Quarterly of Applied Mathematics},
  title     = {Thermo-visco-elasticity at small strains with $L^1$-Data},
  year      = {2009},
  issn      = {0033569X, 15524485},
  number    = {1},
  pages     = {47--71},
  volume    = {67},
  publisher = {Brown University},
  url       = {http://www.jstor.org/stable/43638859},
  urldate   = {2025-02-27},
doi = {10.1090/S0033-569X-09-01094-3}
}

@Article{Shibata1995,
  author     = {Shibata, Yoshihiro},
  journal    = {Math. Methods Appl. Sci.},
  title      = {Global in time existence of small solutions of nonlinear thermoviscoelastic equations},
  year       = {1995},
  issn       = {0170-4214,1099-1476},
  number     = {11},
  pages      = {871--895},
  volume     = {18},
  doi        = {10.1002/mma.1670181104},
  fjournal   = {Mathematical Methods in the Applied Sciences},
  mrclass    = {35Q72 (35B40 73B30 73F15)},
  mrnumber   = {1346664},
  mrreviewer = {Song\ Jiang},
  url        = {https://doi.org/10.1002/mma.1670181104},
}

@Article{Bonetti2003,
  author     = {Bonetti, Elena and Bonfanti, Giovanna},
  journal    = {Electron. J. Differential Equations},
  title      = {Existence and uniqueness of the solution to a 3{D} thermoviscoelastic system},
  year       = {2003},
  issn       = {1072-6691},
  pages      = {No. 50, 15},
  fjournal   = {Electronic Journal of Differential Equations},
  mrclass    = {74H20 (35A15 35Q72 74D10 74F05 74H25)},
  mrnumber   = {1971116},
  mrreviewer = {Song\ Mu\ Zheng},
}

@Article{Rossi2013,
  author     = {Rossi, Riccarda and Roubíček, Tomáš},
  journal    = {Interfaces Free Bound.},
  title      = {Adhesive contact delaminating at mixed mode, its thermodynamics and analysis},
  year       = {2013},
  issn       = {1463-9963,1463-9971},
  number     = {1},
  pages      = {1--37},
  volume     = {15},
  doi        = {10.4171/IFB/293},
  fjournal   = {Interfaces and Free Boundaries. Mathematical Analysis, Computation and Applications},
  mrclass    = {74M15 (35Q74 74A15 74C05)},
  mrnumber   = {3062572},
  mrreviewer = {M.\ Fabrizio},
  url        = {10.4171/IFB/293},
}

@Article{Owczarek2023,
  author  = {Owczarek, Sebastian and Wielgos, Karolina},
  journal = {Mathematical Methods in the Applied Sciences},
  title   = {On a thermo‐visco‐elastic model with nonlinear damping forces and $L^1$ temperature data},
  year    = {2023},
  month   = {02},
  volume  = {46},
  doi     = {10.1002/mma.9098},
}

@Article{Cieslak2023,
  author  = {Cieślak, Tomasz and Muha, Boris and Trifunović, Srđan},
  journal = {Calculus of Variations and Partial Differential Equations},
  title   = {Global weak solutions in nonlinear 3D thermoelasticity},
  year    = {2023},
  month   = {12},
  volume  = {63},
  doi     = {10.1007/s00526-023-02615-2},
}

@Article{Fricke2025,
  author  = {Fricke, Torben J.},
  journal = {Acta Applicandae Mathematicae},
  title   = {Local and Global Solvability in a Viscous Wave Equation Involving General Temperature-Dependence},
  year    = {2025},
  issn    = {1572-9036},
  number  = {1},
  pages   = {1},
  volume  = {200},
  doi     = {10.1007/s10440-025-00751-9},
  refid   = {Fricke2025},
  url     = {https://doi.org/10.1007/s10440-025-00751-9},
}

@unpublished{Winkler,
  author        = {Michael Winkler},
 note = {Preprint},
  title         = {Large-data regular solutions in a one-dimensional thermoviscoelastic evolution problem involving temperature-dependent viscosities},
  year          = {2025},
  archiveprefix = {arXiv},
  eprint        = {2504.20473},
  primaryclass  = {math.AP},
  url           = {https://arxiv.org/abs/2504.20473},
doi = {10.48550/arXiv.2504.20473
}
}

@Unpublished{Claes,
  author = {Leander Claes and Michael Winkler},
  note   = {Preprint},
  title  = {Describing smooth small-data solutions to a quasilinear hyperbolic-parabolic system by W1,p energy analysis},
  doi    = {10.48550/arXiv.2510.21660},
  eprint = {arXiv:2510.21660},
}

@Article{Winklera,
  author   = {Winkler, Michael},
  journal  = {Z. Angew. Math. Phys.},
  title    = {Large-data solutions in one-dimensional thermoviscoelasticity involving temperature-dependent viscosities},
  year     = {2025},
  issn     = {0044-2275,1420-9039},
  number   = {5},
  pages    = {Paper No. 192, 29},
  volume   = {76},
  doi      = {10.1007/s00033-025-02582-y},
  fjournal = {Zeitschrift f\"ur Angewandte Mathematik und Physik. ZAMP. Journal of Applied Mathematics and Physics. Journal de Math\'ematiques et de Physique Appliqu\'ees},
  mrclass  = {74H20 (35D30 35L05 74F05)},
  mrnumber = {4955171},
  url      = {https://doi.org/10.1007/s00033-025-02582-y},
}

@Article{Friesen2024,
  author  = {Friesen, Olga and Claes, Leander and Scheidemann, Claus and Feldmann, N and Hemsel, T and Henning, Bernd},
  journal = {Journal of Physics: Conference Series},
  title   = {Estimation of temperature-dependent piezoelectric material parameters using ring-shaped specimens},
  year    = {2024},
  month   = {09},
  pages   = {012125},
  volume  = {2822},
  doi     = {10.1088/1742-6596/2822/1/012125},
}

@Article{Bies2023,
  author  = {Bies, Piotr and Cieślak, Tomasz},
  journal = {SIAM Journal on Mathematical Analysis},
  title   = {Global-in-Time Regular Unique Solutions with Positive Temperature to One-Dimensional Thermoelasticity},
  year    = {2023},
  month   = {11},
  pages   = {7024-7038},
  volume  = {55},
  doi     = {10.1137/23M1560550},
}

@Article{Gawinecki2016,
  author  = {Gawinecki, Jerzy and Zajaczkowski, Wojciech},
  journal = {Applicationes Mathematicae},
  title   = {On regular solutions to two-dimensional thermoviscoelasticity},
  year    = {2016},
  month   = {01},
  pages   = {1-27},
  volume  = {43},
  doi     = {10.4064/am2299-6-2016},
}

@Article{Gawinecki2016a,
  author  = {Gawinecki, Jerzy and Zajaczkowski, Wojciech},
  journal = {Communications on Pure and Applied Analysis},
  title   = {Global regular solutions to two-dimensional thermoviscoelasticity},
  year    = {2016},
  month   = {02},
  pages   = {1009-1028},
  volume  = {15},
  doi     = {10.3934/cpaa.2016.15.1009},
}

@Article{Pawlow2017,
  author  = {Pawlow, Irena and Zajaczkowski, Wojciech},
  journal = {Communications on Pure and Applied Analysis},
  title   = {Global regular solutions to three-dimensional thermo-visco-elasticity with nonlinear temperature-dependent specific heat},
  year    = {2017},
  month   = {07},
  pages   = {1331-1371},
  volume  = {16},
  doi     = {10.3934/cpaa.2017065},
}
\end{document}